\documentclass[11pt]{amsart}

\usepackage{amsthm, amssymb, amsfonts}
\newtheorem{proposition}{Proposition}[section]
\newtheorem{theorem}{Theorem}[section]
\newtheorem{remark}{Remark}[section]

\usepackage{mathrsfs} 
\usepackage[all]{xy}
\usepackage[capitalize]{cleveref}
\usepackage{multirow}

\newcommand{\teich}{\mathscr{T}}
\newcommand{\hyperbolic}{\mathbb{H}}
\newcommand{\complexes}{\mathbb{C}}

\begin{document}

\title{The horocyclic metrics on Teichm\"{u}ller spaces}
\author{Hideki Miyachi, Ken'ichi Ohshika and Athanase Papadopoulos \medskip}
\address{Hideki Miyachi,
School of Mathematics and Physics,
College of Science and Engineering,
Kanazawa University,
Kakuma-machi, Kanazawa,
Ishikawa, 920-1192, Japan}
\email{miyachi@se.kanazawa-u.ac.jp}  
\address{Ken'ichi Ohshika,
Department of Mathematics,
Gakushuin University,
Mejiro, Toshima-ku, Tokyo, Japan}
  \email{ohshika@math.gakushuin.ac.jp}
\address{Athanase Papadopoulos,
Institut de Recherche Mathématique Avancée,
CNRS et Université de Strasbourg,
7, rue René Descartes, 67084 Strasbourg Cedex, France}
  \email{papadop@math.unistra.fr}
  
\date{\today}
\maketitle

\begin{abstract}
 William Thurston introduced a Finsler metric on the Teichm\"uller spaces of hyperbolic surfaces of finite type which he called the earthquake metric. In this paper, we develop a conformal analogue of this metric. We first define a version of the earthquake metric in the case of the torus and we show that it coincides with the Teichm\"uller metric. Based on this, and using the notion of Teichm\"uller disc, we define a metric on Teichm\"uller spaces of general Riemann surfaces, which we call the horocyclic metric.  Using a complex version of the Legendre transform defined on general Finsler manifolds, we obtain an analogue of Wolpert's duality for the Weil--Petersson symplectic form. This is used to establish a complete analogue in the present conformal setting of the earthquake metric.
 
 This paper will appear in the Annales de l'Institut Fourier.

\medskip

\noindent {\sc Résumé} William Thurston a introduit une métrique de Finsler sur les espaces de Teichmüller des surfaces hyperboliques de type fini, qu’il a appelée la métrique de tremblement de terre. Dans cet article, nous développons un analogue conforme de cette métrique. Nous définissons d’abord une version de la métrique de tremblement de terre dans le cas du tore et montrons qu’elle coïncide avec la métrique de Teichmüller. Partant de là, et en utilisant la notion de disque de Teichmüller, nous définissons une métrique sur les espaces de Teichmüller des surfaces de Riemann générales, que nous appelons la métrique horocyclique.  Utilisant une version complexe de la transformation de Legendre définie sur les variétés de Finsler, nous obtenons un analogue de la dualité de Wolpert pour la forme symplectique de Weil-Petersson. Cela permet d'établir un analogue conforme de la théorie de la distance de tremblement de terre.
 
 Cet article va paraître dans les Annales de l'Institut Fourier.
 
%
%William Thurston a introduit une norme sur l'espace tangent en chaque point de l'espace de Teichm\"uller d'une surface hyperbolique qu'il a appelée norme de tremblement de terre. Cette norme induit une métrique de Finsler sur l'espace de Teichm\"uller, appelée \emph{métrique du tremblement de terre}. Nous étudions un analogue de cette métrique du point de vue conforme. 
%Dans le cas du tore sans point marqué, l'analogue de la métrique du tremblement de terre est naturellement appelée la métrique horocyclique. Dans le cas des surfaces de Riemann de caractéristique d'Euler négative, la collection des métriques horocycliques sur les disques de Teichm\"uller définit une métrique sur l'espace de Teichm\"uller, que nous appelons aussi métrique horocyclique. Nous montrons que cette métrique coïncide avec la métrique de Teichm\"uller.  Une version complexe de la transformée de Legendre définie pour les variétés Finsleriennes donne un analogue de la dualité de Wolpert pour la forme symplectique de Weil--Petersson, ce qui établit un analogue complet de la théorie de Thurston de la norme de tremblement de terre dans le cadre conforme.
% 
\medskip
\noindent AMS classification:  30F60, 32G15, 30F10.

\medskip
\noindent Keywords: 
 Teichm\"uller space, quadratic differential,   horocyclic metric, earthquake flow, Fenchel--Nielsen twist, horocyclic flow, horocyclic norm,
 Teichm\"{u}ller norm, Teichm\"uller metric, Weil--Petersson metric, Wolpert duality, Teichm\"uller disc, extremal length, Legendre transform.

\end{abstract}
\section{Introduction}

William Thurston introduced a new family of deformations of hyperbolic surfaces which generalise the Fenchel--Nielsen deformations. He called these deformations (left)\footnote{In this paper, earthquakes will always be left earthquakes.} earthquakes and proved that any two points in Teichm\"{u}ller space can be joined by an earthquake.
In the paper \cite{KeA}, Kerckhoff obtained an infinitesimal version of this result, namely, that any tangent vector to Teichm\"{u}ller space can be expressed as an infinitesimal earthquake deformation.
Thurston mentioned in \cite{ThM} that this expression induces a norm on the tangent spaces of Teichm\"{u}ller space, which he called the \emph{earthquake norm}, hence also a Finsler metric on that space.
This metric was studied by Huang--Ohshika--Pan--Papadopoulos in \cite{HOPP} and it was called there the earthquake metric. 
We note, among the properties proved in that paper, a duality result, which was also hinted by Thurston in his paper \cite{ThM}, namely, that the tangent space at a point of Teichm\"{u}ller space equipped with the earthquake norm  is linearly isometric to the cotangent space at the same point equipped with Thurston's conorm, that is, the norm dual to the one associated with Thurston's asymmetric metric, via the Weil--Petersson symplectic form.

We first consider the case of the torus, that is, the closed surface of genus one. To define the Teichm\"{u}ller space of this surface, we consider Euclidean structures instead of hyperbolic structures.
What corresponds to the earthquake deformation in this case is naturally called the horocyclic deformation. Indeed, in this Euclidean setting, the Teichm\"uller space of the surface is the hyperbolic plane, and the earthquake paths are the horocycles of this plane. 
In the same way as the earthquake metric was defined in the hyperbolic case, we introduce a Finsler metric on the Teichm\"{u}ller space of the torus, using these infinitesimal horocyclic deformations. We call this Finsler metric the horocyclic metric. In the first part of this paper, we prove that the horocyclic metric is isometric (up to the constant $2$) to  the natural hyperbolic metric on the Teichm\"{u}ller space of tori.
We also show that the Weil--Petersson symplectic form on the Teichm\"{u}ller space of tori establishes a duality between the cotangent space to this space equipped with the conorm associated with the Teichm\"{u}ller metric and the tangent space equipped with the horocyclic norm.
Since both the Teichm\"{u}ller metric and the horocyclic metric are isometric to $2$ $\times$ the hyperbolic metric, this duality looks like a \lq self-duality'. 

We pass then to the Teichm\"uller spaces of arbitrary surfaces of  finite type, that is, closed surfaces of finite genus with a finite number of marked points. Making use of Teichm\"{u}ller discs, we define horocyclic deformations\footnote{In this paper, notions such as \emph{horocyclic deformation} and \emph{horocyclic vector field}, inspired by Thurston's paper \cite{ThM}, are different from other notions of \emph{horocyclic flow}  used in the literature, e.g. the classical notion of horocycle flow in \cite{Ratner}, and the one introduced by Masur on the moduli space of quadratic differentials on a surface \cite{Masur-Ergodic}, although all these notions are related.} of these surfaces and horocyclic metrics for these general Teichm\"{u}ller spaces. Such a metric is also defined via a family of norms on tangent spaces. 
We show that these norms induce  the familiar Teichm\"{u}ller metric, that is, the metric that coincides with half of the hyperbolic metric on each Teichm\"{u}ller disc \cite{Roy}. Thus, in the case of surfaces of arbitrary finite type, the Teichm\"uller and the horocyclic metrics coincide.

In Finsler geometry, there is a notion of Legendre transform, which is a map from the cotangent vector space to the tangent vector space at the same point of a manifold with a Finsler structure. This notion was introduced and studied by Ohta--Sturm \cite{Ohta}.
Applying a complex version  to Teichm\"{u}ller space equipped with the Teichm\"{u}ller metric, it turns out that the Legendre transform coincides with the infinitesimal Teichm\"{u}ller homeomorphism.
This implies in turn a duality between the infinitesimal length function and the infinitesimal horocyclic deformation.

Table 1 summarizes part of the analogy between Thurston's theory of the earthquake metric and its conformal analogue.

\begin{table}[h!]
\begin{tabular}{|c|c|}
\hline
Hyperbolic case & Conformal case \\
\hline\hline
Thurston asymmetric metric & Teichm\"uller metric \\
\hline
\multirow{2}{*}{Earthquake metric} & 
Horocyclic metric  $=$ \\
& $2 \times$ Teichm\"uller metric \\
\hline
Hyperbolic length & Square root of extremal length  \\
\hline
Infinitesimal earthquake deformation 
& Infinitesimal horocyclic deformation \\
\hline
\end{tabular}

\bigskip
\caption{\smaller In each case, a duality between the gradient vectors of the hyperbolic length (resp. extremal length) functions and the infinitesimal earthquake deformations (resp. horocyclic deformations) holds. However, in the hyperbolic case, the duality is given by a linear isometry, but in the conformal case,  the duality is neither linear nor isometric.}
\end{table}

\noindent \emph{Acknowledgements.}
This work was done during the authors' stay in the Centre International de Rencontres Math\'{e}mathique at Luminy, Marseille, for a project of research in residence.
The authors express their sincere gratitude to CIRM for its hospitality and support, as well as to the anonymous referee of this article whose comments helped clarifying several points. 

\section{Preliminaries}
\subsection{Teichm\"uller space}
\label{subsec:teichmuller_space}
Let $\Sigma_{g,m}$ be an oriented closed surface of genus $g\geq 0$ with $m\geq 0$ marked points. 
A \emph{marked Riemann surface} of type $(g,m)$ is a pair $(M,f)$ consisting of a Riemann surface $M$ of genus $g$ with $m$ distinguished points and an orientation-preserving homeomorphism $f\colon \Sigma_{g,m}\to M$ which sends the marked points of $\Sigma_{g,m}$ to the distinguished points of $M$.
Two marked Riemann surfaces $(M_1,f_1)$ and $(M_2,f_2)$ are said to be \emph{Teichm\"uller equivalent} if there is a biholomorphic map $h\colon M_1\to M_2$ sending the distinguished points to themselves such that $h\circ f_1$ is homotopic to $f_2$ fixing the marked points.
The \emph{Teichm\"uller space $\teich_{g,m}$ of type $(g,m)$} is the set of Teichm\"uller equivalence classes of  marked Riemann surfaces. 
In the case when $m=0$, we denote $\teich_{g,0}$ by $\teich_g$. 

The \emph{Teichm\"uller distance} $d_T$ on $\teich_{g,m}$ is defined by
$$
d_T((M_1,f_1),(M_2,f_2))=\frac{1}{2}\log \inf_hK(h)
$$
where $h$ ranges over the quasiconformal maps $h\colon M_1\to M_2$ which preserve the sets of distinguished points and which are homotopic to $f_2\circ f_1$ fixing the distinguished points, and where $K(h)$ denotes the maximal dilatation of $h$.

\subsection{Measured foliations}
\label{subsec:MF}
We think of $\Sigma_{g,m}$ as a pair of an oriented closed surface $\Sigma_g$ of genus $g$ together with a set of marked points $D\subset \Sigma_g$.
A \emph{measured foliation} $\lambda$ on $\Sigma_{g,m}$ with singularities of order $k_1$, $\cdots$, $k_n$ at $x_1$, $\cdots$, $x_n\in \Sigma_g$ ($k_i\ge -1$ if $x_i\in D$, and $k_i\ge 0$ otherwise) is given by a locally finite open cover $\{U_i\}_{i\in I}$ of $M\setminus (\{x_1,\cdots,x_n\}\cup D)$ and a non-vanishing $C^\infty$ real-valued closed $1$-form $\varphi_i$ on each $U_i$ such that
\begin{itemize}
\item[(a)] $\varphi_i=\pm \varphi_j$ on $U_i\cap U_j$; and
\item[(b)] at each $x_i$, assumed to be contained in $U_j$, there is a local chart $(u,v)\colon V\to \mathbb{R}^2$ such that for $z=u+iv$, $\varphi={\rm Im}(z^{k_i/2}dz)$ on $V\cap U_j$ for some branch of $z^{k_i/2}$ in $V \cap U_j$.
\end{itemize}
The set of such pairs $\{(U_i,\varphi_i)\}_{i\in I}$
 is called an atlas for $\lambda$. For $t>0$, we denote by $t\lambda$ the measured foliation with atlas $\{(U_i,t\varphi_i)\}_{i\in I}$. 
We note that  by the condition (a), we have  $|\varphi_i|=|\varphi_j|$ on $U_i\cap U_j$.
For a smooth path $\gamma$ on $M$, we define
$$
\lambda(\gamma)=\sum_{i\in I}\int_{U_i\cap\gamma}e_i|\varphi_i|,
$$
where $\{e_i\}_i$ is a partition of unity subordinate to $\{U_i\}_{i\in I}$.

Let $\mathcal{S}$ be the set of homotopy classes of simple closed curves on $\Sigma_{g,m}$ which are non-contractible and non-peripheral (that is, non-homotopic to a marked point).
For $\alpha\in \mathcal{S}$, we define the \emph{intersection number} $i(\alpha,\lambda)$ by
$$
i(\alpha,\lambda)=\inf_{c\in \alpha}\lambda(c).
$$
Two measured foliations $\lambda_1$ and $\lambda_2$ are equivalent if $i(\alpha,\lambda_1)=i(\alpha,\lambda_2)$ for all $\alpha\in \mathcal{S}$. To simplify, an  equivalence class is also called a \emph{measured foliation} on $\Sigma_{g,m}$, and the set of (equivalence classes of) measured foliations is denoted by $\mathcal{MF}$. 
For  simplicity, we also denote by $\lambda$ the equivalence class of $\lambda$.  

There is a geometric description of  measured foliations, which we recall now.

By the condition (a) in the definition of measured foliation, the foliation $F_\lambda$ whose leaves are integral curves of  the kernels of all $\varphi_i$ is well defined. 
The foliation $F_\lambda$ is thought of as a singular foliation on the surface, with singularities at $x_i$, $\cdots$, $x_n$. 
By the condition (b) in the definition of measured foliation, at each $x_i$, $F_\lambda$ has a $k_i+2$-prong singularity, that is, there are $k_i+2$ leaves emanating from $x_i$.
Furthermore, $|\varphi_i|$ induces a transverse measure on $F_\lambda$.
%
%The \emph{saddle connection} of $\lambda$ is a leaf of $F_{\lambda}$ connecting critical points. 
%For measured foliations $\lambda_1$ and $\lambda_2$, $\lambda_1$ is related to $\lambda_2$ by a \emph{Whitehead move}, denote by $\lambda_1\to \lambda_2$,  if there is a saddle connection of $\lambda_1$ such that $F_{\lambda_2}$ is isotopic to $F_{\lambda_1}$ after collapsing the saddle connection. 
%It is known that $\lambda_1$ and $\lambda_2$ are equivalent if and only if there is a sequence of measured foliations $\lambda_1=\mu_1$, $\mu_2$, $\cdots$, $\mu_k=\lambda_2$ such that for each $i=1$, $\cdots$, $n-1$, either $\lambda_i\to \lambda_{i+1}$ or $\lambda_{i+1}\to \lambda_i$.

%Let $\mathcal{MF}$ be the set of measured foliations.
There is a topology on the space $\mathcal{MF}$, defined in such a way that a sequence $(\lambda_n)_{n=1}^\infty$ in $\mathcal{MF}$ converges to $\lambda\in \mathcal{MF}$ if and only if for any $\alpha\in \mathcal{S}$, the sequence $(i(\alpha,\lambda_n))_{n=1}^\infty$ tends to $i(\alpha,\lambda)$ as $n\to\infty$.
There is a continuous action of the multiplicative group $\mathbb{R}_{>0}$ on $\mathcal{MF}$, defined by
$$
\mathbb{R}_{>0}\times \mathcal{MF}\ni (t,\lambda)\to t\lambda\in \mathcal{MF}.
$$
Thurston showed that $\mathcal{MF}$ is homeomorphic to $\mathbb{R}^{6g-6+2m}\setminus \{0\}$,
and that the quotient space by the $\mathbb{R}_{>0}$-action is homeomorphic to $\mathbb{S}^{6g-7+2m}$.

%Let $\mathcal{S}$ be the homotopy classes of non-contractible and non-peripheral simple closed curves on $\Sigma_{g,m}$.
%We define  $\mathcal{WS}=\{t\alpha\mid t>0, \alpha\in \mathcal{S}$.
%We identify $\mathcal{S}$ as a subset of $\mathcal{WS}$ by setting $1\cdot \alpha=\alpha$. 
%The \emph{intersection number} on $\mathcal{WS}$ is defined by
%$$
%i(t\alpha,s\beta)=ts\,i(\alpha,\beta)
%$$
%where $i(\alpha,\beta)$ on  the right-hand side is the geometric intersection number which is defined by the minimum number of the intersection points between the representatives of $\alpha$ and $\beta$.
%The closure $\mathcal{MF}$ of the embedding
%$$
%\mathcal{WS}\ni t\alpha\mapsto [\beta\mapsto i(t\alpha,\beta)]\in \mathbb{R}_{\ge 0}^{\mathcal{S}}
%$$
%is called the space of \emph{measured foliations} on $\Sigma_{g,m}$, where we define a topology on the function space $\mathbb{R}_{\ge 0}^{\mathcal{S}}$ induced from  the pointwise convergence.

\subsection{Infinitesimal structures}
Let $x=(M,f)$ be a point in $\teich_{g,m}$.
Let $L^\infty(M)$ be the complex Banach space of bounded measurable $(-1,1)$-forms $\mu=\mu(z)dz^{-1}d\overline{z}$ equipped with the norm
$$
\|\mu\|_\infty={\rm ess.sup}_{z\in M}|\mu(z)|.
$$
Let $\mathcal{Q}_x$ be the complex Banach space of integrable holomorphic quadratic differentials $q=q(z)dz^2$ on $M \setminus f(D)$ equipped with the norm
$$
\|q\|=\iint_M|q(z)|dxdy\quad (z=x+iy).
$$
There is a natural pairing between $L^\infty(M)$ and $\mathcal{Q}_x$ defined, for $\mu=\mu(z)dz^{-1}d\overline{z}\in L^\infty(M)$ and $q=q(z)dz^2\in \mathcal{Q}_x$, by
$$
\langle\!\langle \mu,q\rangle\!\rangle=\iint_{M}\mu(z)q(z)dxdy\quad (z=x+iy).
$$
Teichm\"uller showed that the (holomorphic) tangent space $T_x\teich_{g,m}$ is identified with the quotient space
$$
L^\infty(M)/\{\mu\in L^\infty(M)\mid 
\mbox{$\langle\!\langle \mu,q\rangle\!\rangle=0$
for all $q\in \mathcal{Q}_x$}\}.
$$
We denote by $[\mu]\in T_x\teich_{g,m}$ the equivalence class of $\mu\in L^\infty(M)$.
The above pairing descends to the non-degenerate pairing
$$
T_x\teich_{g,m}\times \mathcal{Q}_x\ni (v=[\mu],q)\mapsto \langle v,q\rangle=\langle\!\langle \mu,q\rangle\!\rangle.
$$
Teichm\"uller's theorem tells us that for $x\in \teich_{g,m}$ and for any $v\in T_x\teich_{g,m}$, there are unique $q\in \mathcal{Q}_x$ and $k\ge 0$ such that
$v=[k\overline{q}/|q|]$.

By virtue of the pairing between $T_x\teich_{g,m}$ and $\mathcal{Q}_x$, the space $\mathcal{Q}_x$ is naturally identified with the holomorphic cotangent space $T^*_x\!\teich_{g,m}$ at $x\in \teich_{g,m}$.
Let $\mathcal{Q}_{g,m}=\cup_{x\in \teich_{g,m}}\mathcal{Q}_x$ be the holomorphic vector bundle of holomorphic quadratic differentials over $\teich_{g,m}$. Then,  the above discussion implies that $\mathcal{Q}_{g,m}$ is naturally regarded as the holomophic cotangent bundle over $\teich_{g,m}$.

\subsection{The Teichm\"uller metric}
For $x\in \teich_{g,m}$ and $v\in T_x\teich_{g,m}$, we define the \emph{Teichm\"uller norm} of $v$ by
$$
\kappa(v)=\sup\{{\rm Re}\langle v,q\rangle\mid \mbox{$q\in \mathcal{Q}_x$, $\|q\|=1$}\}.
$$
The Teichm\"{u}ller norm induces a Finsler metric on $\teich_{g,m}$, which is called the Teichm\"{u}ller metric.
Royden \cite{Roy} showed that the Teichm\"uller metric coincides with the Kobayashi metric on $\teich_{g,m}$.
We can easily see that 
\begin{equation}
\label{eq:teichmuller-beltrami}
\kappa\left(\left[\frac{\overline{q}}{|q|}\right]\right)=1
\end{equation}
for any non-zero $q\in \mathcal{Q}_{g,m}$.

\section{The case of the torus}
In this section, we study the horocyclic deformations in the Teichm\"{u}ller space of tori, $\teich_1$.
\subsection{The Teichm\"uller space of tori}
\label{subsec:teichmuller_space_of_torI}
Fix a generator pair $\{A,B\}$ of the first homology group $H_1(\Sigma_1,\mathbb{Z})$ of $\Sigma_1$ such that the algebraic intersection number $A\cdot B$ is $+1$. For $x=(M,f)\in \teich_1$, we take  a holomorphic 1-form $\omega_x$ on $M$. Then, the following map
\begin{equation}
\label{eq:period}
\teich_1\ni x\mapsto \tau(x)=\left.\int_B\omega_x\right/\int_A\omega_x\in \mathbb{H}=\{\tau\in \mathbb{C}\mid
{\rm Im}(\tau)>0\}
\end{equation}
is well defined and is called the period map. Furthermore, it is known that the Teichm\"uller distance on $\teich_1$   coincides with the Poincaré distance $d_{\mathbb{H}}$ of curvature $-4$  on $\mathbb{H}$ via the map
$$
d_T(x_1,x_2)=d_{\mathbb{H}}(\tau(x_1),\tau(x_2)).
$$
The identification between $\teich_1$ and $\mathbb{H}$ via the period map \eqref{eq:period} defines a complex structure on $\teich_1$. 
On the other hand, $\teich_1$ admits a natural complex structure inherited from the complex Banach space of complex coefficients of quasiconformal maps on a (fixed) complex torus. 
These two coincide.

In what follows, we identify the Teichm\"uller space $\teich_1$ with the upper half plane $\mathbb{H}$.

\subsection{Fenchel--Nielsen coordinates on $\teich_1$}
For the contents of this subsection and the next one, we refer the reader to Imayoshi--Taniguchi \cite[\S7.3.5]{IT}. 

Take $x=(M,f)\in \teich_1$, and  denote $\tau(x)$ by $\tau$. Then, the complex torus $M$ is biholomorphic to $M_\tau=\mathbb{C}/\mathbb{Z}\oplus \tau\mathbb{Z}$ via the Abel--Jacobi map (after identifying the complex torus $M_\tau$ with the Jacobian variety of $M$) and the marking $f\colon \Sigma_1\to M$ corresponds to an orientation-preserving homeomorphism $f_\tau\colon \Sigma_1\to M_\tau$  which sends the generators $A$ and $B$ of $H_1(\Sigma_1,\mathbb{Z})$ to the generators $1$ and $\tau$ of the lattice $\mathbb{Z}\oplus \tau\mathbb{Z}$. 

Let $\tau$ be a point in $\mathbb{H}\cong \teich_1$. We define the Fenchel--Nielsen coordinates $(\ell,\theta)$ on $\teich_1$ as follows. Fix a simple closed curve $\alpha$ on $\Sigma_1$ whose homology class is equal to $A$. 
Consider $ds_{\tau}=|dz|/\sqrt{{\rm Im}(\tau)}=\rho_\tau |dz|$,  the flat metric on $M_\tau$ of unit area, where the $z$-coordinate  is that of the universal covering space $\mathbb{C}$ of $M_\tau=\mathbb{C}/\mathbb{Z}\oplus \tau\mathbb{Z}$. The length $\ell(\tau)$ of the geodesic representative of $\alpha$ on $(M_\tau,ds_{\tau})$ is equal to $1/\sqrt{{\rm Im}(\tau)}$. The function $t(\tau)=-{\rm Re}(\tau)/\sqrt{{\rm Im}(\tau)}$ represents the twist parameter of unit speed. Set $\theta(\tau)=2\pi t(\tau)/\ell(\tau)$. Then, the \emph{Fenchel--Nielsen coordinates} of $\teich_1=\mathbb{H}$ with respect to the curve $\alpha$ are given by
\begin{equation}
\label{eq:FN}
\teich_1=\mathbb{H}\ni \tau\mapsto (\ell(\tau),\theta(\tau))
=\left(\frac{1}{\sqrt{{\rm Im}(\tau)}},-2\pi {\rm Re}(\tau)\right)
\in \mathbb{R}_{>0}\times \mathbb{R}.
\end{equation}

\subsection{The Weil--Petersson metric on $\teich_1$ and the Wolpert formula}
%Let $\partial/\partial \tau|_{\tau}$ be a generator of the holomorphic tangent vector in $T_\tau\teich_1=T_\tau\mathbb{H}$. 
The tangent vector $(\partial/\partial \tau)_{\tau}$ to $\teich_1$ represents the infinitesimal deformation $M_\tau\to M_{\tau+t}$ as $t \in \complexes$ goes to $0$.
The affine deformation (the deformation by the extremal quasiconformal map) $f_t\colon M_\tau\to M_{\tau+t}$ is given by the affine map
$$
\tilde{f}_t(z)=\left(1+\frac{t}{\tau-\overline{\tau}}\right)z+\frac{-t}{\tau-\overline{\tau}}\overline{z}\quad (z\in \mathbb{C})
$$
which is equivariant under the actions of the lattices representing $M_\tau$ and $M_{\tau+t}$. The Beltrami differential $\mu(t)$ of $\tilde{f}_{t}$ behaves as
$$
\mu(t)=\frac{-1}{\tau-\overline{\tau}}t+o(t)
$$
as $t\to 0$. Hence, the tangent vector $(\partial/\partial \tau)_{\tau}$ is represented by the infinitesimal deformation of the (infinitesimal) complex coefficient
$$
\frac{-1}{\tau-\overline{\tau}}\frac{d\overline{z}}{dz}.
$$
We define the \emph{scalar (Hermitian) product} $h_{WP}$ on $T_\tau\teich_1=T_\tau\mathbb{H}$ by
$$
h_{WP}\left(
\alpha\left(\frac{\partial}{\partial \tau}\right)_{\tau},
\beta\left(\frac{\partial}{\partial \tau}\right)_{\tau}
\right)=\int_{M_\tau}\frac{\alpha\overline{\beta}}{|\tau-\overline{\tau}|^2}\cdot \rho_\tau^2dxdy
=\frac{\alpha\overline{\beta}}{4{\rm Im}(\tau)^2}
\quad
(\alpha, \beta\in \mathbb{C}).
$$
The \emph{Weil--Petersson metric} $ds_{WP}^2$ on the Teichm\"uller space $\teich_1=\mathbb{H}$ is defined by
\begin{equation}
\label{eq:WP_metric_uv}
ds_{WP}^2=2{\rm Re}(h_{WP})=\frac{1}{2{\rm Im}(\tau)^2}|d\tau|^2,
\end{equation}
which is twice the hyperbolic metric of curvature $-4$ on $\mathbb{H}$. The fundamental (K\"ahler) form $\omega_{WP}$ of $ds^2_{WP}$ is given by
\begin{equation}
\label{eq:WP_funcamental_form_uv}
\omega_{WP}=-2{\rm Im}(h_{WP})
%=\frac{i}{4{\rm Im}(\tau)^2}d\tau\wedge d\overline{\tau}
=\frac{id\tau\wedge d\overline{\tau}}{4{\rm Im}(\tau)^2}.
\end{equation}

In what follows, we set  $\tau=u+iv$.
Under the Fenchel--Nielsen coordinates discussed in the previous section, 
since $\ell=1/\sqrt{v},$  $t=-u/\sqrt{v}$ and $(i/2)d\tau\wedge d\overline{\tau}=du\wedge dv$, the fundamental K\"ahler form $\omega_{WP}$ is expressed as \begin{equation}
\label{eq:WP_funcamental_form}
\omega_{WP}=dt\wedge d\ell,
\end{equation}
which is regarded as the Wolpert formula in this torus setting.

\subsection{Horocyclic deformations}

\subsubsection{Measured foliations on the torus $\Sigma_1$}
We use the symbols defined in \S\ref{subsec:teichmuller_space_of_torI} frequently.
Henceforth, any rational number is written in the form $p/q$ where $p$ and $q$ are relatively prime and $q>0$. 
We define $\hat{\mathbb{Q}}$ to be $\mathbb{Q}\cup\{\infty\}$. 
The point $\infty$ is  denoted by $1/0$.

For $p/q\in \hat{\mathbb{Q}}$, the \emph{$p/q$-curve} $\alpha_{p/q}$ is the homotopy class of the simple closed curves representing the homology class $pA+qB$. It is known that any simple closed curve on $\Sigma_1$ is homotopic to a curve in $\alpha_{p/q}$ for some $p/q\in \hat{\mathbb{Q}}$.
If we identify $\Sigma_1$ with $\mathbb{C}/\mathbb{Z}\oplus i\mathbb{Z}$ ($i$ is the imaginary unit) as differentiable manifolds, the $p/q$-curve is thought of as a measured foliation (cf. \S\ref{subsec:MF}) with the differential form $\varphi_{p/q}=qdx-pdy$ where $z=x+iy$ is the coordinate system of the universal cover $\mathbb{C}$ of $\Sigma_1=\mathbb{C}/\mathbb{Z}\oplus i\mathbb{Z}$.

The geometric intersection $i(\alpha_{p/q},\alpha_{r/s})$ between the $p/q$-curve and the $r/s$-curve is 
$$
i(\alpha_{p/q},\alpha_{r/s})=\left|\iint_{\Sigma_1}\varphi_{p/q}\wedge \varphi_{r/s}\right|=|ps-rq|.
$$

Let $\mathcal{WS}=\{t\alpha_{p/q}\mid p/q\in \hat{\mathbb{Q}}\}$ be the set of (nonzero homotopy classes of) weighted simple closed curves on the torus. We define the intersection number between the weighted curves $t\alpha_{p/q}$ and $s\alpha_{r/s}$ by
\begin{equation}
\label{eq:intersection_numner_weighted}
i(t\alpha_{p/q},t'\alpha_{r/s})=tt'|ps-rq|=\left|
\det\begin{bmatrix} tp & t'r \\ tq & t's \end{bmatrix}\right|.
\end{equation}
We embed $\mathcal{WS}$ into $\mathbb{R}_{\ge 0}^{\hat{\mathbb{Q}}}$ by
$$
\mathcal{WS}\ni t\alpha_{p/q}\mapsto [\alpha_{r/s}\mapsto t\,i(\alpha_{p/q},\alpha_{r/s})]\in \mathbb{R}_{\ge 0}^{\hat{\mathbb{Q}}}.
$$
%We provide the function space $\mathbb{R}_{\ge 0}^{\hat{\mathbb{Q}}}$ with the topology induced from the pointwise convergence topology. 
As in the general description, the measured foliation space  $\mathcal{MF}=\mathcal{MF}(\Sigma_1)$ is the closure of of the image of the embedding of $\mathcal{WS}$.
By Euler--Poincar\'{e}'s  formula, measured foliations on $\Sigma_1$ do not have singularities.
% is called the \emph{space of measured foliations} on $\Sigma_1$. There is a natural action of the multiplicative group $\mathbb{R}_{>0}$ on the space $\mathcal{MF}\subset \mathbb{R}_{\ge 0}^{\hat{\mathbb{Q}}}$, by multiplication of the transverse measure by a constant. The quotient space $\mathcal{PMF}=\mathcal{PMF}(\Sigma_1)$ of $\mathcal{MF}\setminus \{0\}$ by this action is called the \emph{ projective measured foliation space}.

The group $\mathbb{Z}_2$ of order $2$ acts naturally on the plane $\mathbb{R}^2$ by the $\pi$-rotation centred at the origin. We denote by $[x,y]\in \mathbb{R}^2/\mathbb{Z}_2$ the equivalence class of a point $(x,y)\in \mathbb{R}^2$. Then there is a natural embedding
\begin{equation}
\label{eq:emb_WS_to_RZ2}
\mathcal{WS}\ni t\alpha_{p/q}\mapsto [tp,tq]\in \mathbb{R}^2/\mathbb{Z}_2.
\end{equation}

A $p/q$-curve $\alpha_{p/q}$ corresponds to the element $[p,q]\in \mathbb{R}^2/\mathbb{Z}_2$.
The intersection number function given in \eqref{eq:intersection_numner_weighted} extends continuously to the product of two $\mathbb{R}^2/\mathbb{Z}_2$. Hence, the embedding \eqref{eq:emb_WS_to_RZ2} extends to a homeomorphism
$$
\mathcal{MF}\to \mathbb{R}^2/\mathbb{Z}_2.
$$
Let $\lambda_{[a,b]}\in \mathcal{MF}$ be the measured foliation corresponding to the element $[a,b]\in \mathbb{R}^2/\mathbb{Z}_2$. 
By definition, we have the equality $\lambda_{[tp,tq]}=t\alpha_{p/q}$ as measured foliations.
From the definition, we see that a sequence of weighted simple closed curves $t_n\alpha_{p_n/q_n}$ converges to a measured lamination $\lambda_{[a,b]}$ in $\mathcal{MF}$ if and only if $[t_np_n,t_nq_n]\to [a,b]$ in $\mathbb{R}^2/\mathbb{Z}_2$.

The action of $\mathbb{R}_{>0}$ on $\mathcal{MF}$ corresponds to the multiplication operation on $\mathbb{R}/\mathbb{Z}_2$. Hence, the homeomorphism $\mathcal{MF}\to \mathbb{R}^2/\mathbb{Z}_2$ induces a homeomorphism $\mathcal{PMF}\to \mathbb{R}\cup\{\infty\}=\mathbb{RP}^1$.
We choose a homeomorphism sending the projective class of the $p/q$-curve to $-p/q$ (rather than  $p/q$), which is more convenient for our purpose as we shall explain later.
Thus we have  the following commutative diagram.
\begin{equation}
\label{eq:commute_PMF}
\xymatrix@C=20pt{
\mathcal{MF}\setminus\{0\} \ar[d]_{proj}\ni \lambda_{[x,y]} \ar[rrd] \\
\mathcal{PMF}\ar@<-0.5ex>[rr]^-{\cong} & &\qquad -\frac{x}{y}  \in \mathbb{R}\cup\{\infty\}=\mathbb{RP}^1.
}
\end{equation}

The \emph{length} $\ell_{[p,q]}(\tau)$ of the $p/q$-curve $\alpha_{p/q}$ on the genus-1 Riemann surface $M_\tau$ is the length of the geodesic representative of the curve $f_\tau(\alpha_{p/q})$ with respect to the area-1 flat metric $ds_\tau$. 
By calculation, we have
$$
\ell_{[p,q]}(\tau)=\frac{|p+q\tau|}{\sqrt{{\rm Im}(\tau)}}.
$$
We note that each level curve of the length function $\ell_{[p,q]}$ is a circle (with one point missing) in $\mathbb{H}$ tangent to $-p/q \in \mathbb R$, and that $\ell_{[p,q]}(\tau)$ goes to $0$ when $\tau$ tends to  $-p/q$ non-tangentially. 
This is the reason why we chose to send  $[x,y]\in \mathcal{MF}\setminus\{0\}$  to $-x/y$ in \eqref{eq:commute_PMF}.
We define the length of the weighted curve $t\alpha_{p/q}$ by
$\ell_{[tp,tq]}(\tau)=t\ell_{[p,q]}(\tau)$.
If a sequence of weighted simple closed curves $t_n\alpha_{p_n/q_n}$ converges to $\lambda_{[a,b]}\in \mathcal{MF}$ as $n\to\infty$, then we have
$$
\ell_{[t_np_n,t_nq_n]}(\tau)=t_n\ell_{[p_n,q_n]}(\tau)=\frac{|t_np_n+t_nq_n\tau|}{\sqrt{{\rm Im}(\tau)}}\to \frac{|a+b\tau|}{\sqrt{{\rm Im}(\tau)}},
$$
as $n\to \infty$. Hence, it is natural to define the length of $\lambda_{[a,b]}\in \mathcal{MF}$ by
\begin{equation}
\label{eq:length_function_ab}
\ell_{[a,b]}(\tau)=\frac{|a+b\tau|}{\sqrt{{\rm Im}(\tau)}}.
\end{equation}

\subsubsection{Fenchel--Nielsen and horocyclic deformations}
In this subsection, we consider the Fenchel--Nielsen and the horocyclic deformations of the flat surfaces $M_\tau$.

Consider an element $p/q\in \hat{\mathbb{Q}}$. Fix a geodesic representative $\alpha^*_{p/q}$ of $\alpha_{p/q}$ on $M_\tau$. Cut $M_\tau$ along $\alpha^*_{p/q}$, and glue it back after twisting to the left one of the boundaries with respect to the other by the distance $t\in \mathbb{R}$. We call the deformation the \emph{Fenchel--Nielsen deformation of $M_\tau$ of length $t$ along the $p/q$-curve}.  We denote by $E_{t\alpha_{p/q}}(\tau)\in \mathbb{H}\cong \teich_1$ the parameter of the resulting marked torus obtained by the deformation described above. 
In the case when $[p,q]=[1,0]$, the $1/0$-curve $\alpha_{1/0}$ is represented by the horizontal line in the universal cover $\mathbb{C}$ of $M_\tau$.
Therefore, the Fenchel-Nielsen deformation $E_{t\alpha_{1/0}}$ is given by
\begin{equation}
E_{t\alpha_{1/0}}(\tau)=\tau-\frac{t}{\ell_{[1,0]}(\tau)},
\end{equation}
where the minus sign of the second term comes from that of the twist parameter in \eqref{eq:FN}.
By conjugating the deformation along the $1/0$-curve by  $\begin{bmatrix}
s & r \\
q & p \\
\end{bmatrix}\in {\rm PSL}_2(\mathbb{Z})$ (with $s,r$ satisfying $ps-rq=1$) taking  $-p/q$ to $\infty$, which corresponds to the change of marking in the Teichm\"uller space $\teich_1$,
the Fenchel-Nielsen deformation $E_{t\alpha_{p/q}}$ is obtained as
\begin{equation}
E_{t\alpha_{p/q}}(\tau)=\frac{(1-tpq/\ell_{[p,q]}(\tau))\tau-tp^2/\ell_{[p,q]}(\tau)}{(tq^2/\ell_{[p,q]}(\tau))\tau+(1+tpq/\ell_{[p,q]}(\tau))}
\end{equation}
on $\mathbb{H}$.  
At the parameter $t=\ell_{[p,q]}(\tau)$, the point $E_{\ell_{[p,q]}(\tau)\alpha_{p/q}}(\tau)\in \mathbb{H}$ corresponds to the image of the action of the right-handed Dehn twist along the curve $\alpha_{p/q}$ of $(M_\tau,f_\tau)$ on $\teich_1$. If  $(t_n\alpha_{p_n/q_n})$ converges to  $\lambda_{[a,b]}$ in $\mathcal{MF}$, then
$$
\frac{t_np_nq_n}{\ell_{[p_n,q_n]}(\tau)}=
\frac{t_np_n\cdot t_nq_n}{\ell_{[t_np_n,t_nq_n]}(\tau)}
\to \frac{ab}{\ell_{[a,b]}(\tau)}
$$
as $n\to \infty$.
Therefore, we have a continuous (smooth) action of $\mathcal{MF}$ on $\mathbb{H}\cong \teich_1$ expressed as
\begin{align}
&\mathcal{MF}\times \mathbb{H}\ni (\lambda_{[a,b]},\tau)
\mapsto 
\nonumber \\
&\qquad E_{\lambda_{[a,b]}}(\tau):=\frac{(1-ab/\ell_{[a,b]}(\tau))\tau-a^2/\ell_{[a,b]}(\tau)}{b^2/\ell_{[a,b]}(\tau)\tau+(1+ab/\ell_{[a,b]}(\tau))}
\in \mathbb{H}.
\label{eq:earthquake_deformation_ab}
\end{align}
By definition, $E_{\lambda_{[ta,tb]}}(\tau)=E_{t\lambda_{[a,b]}}(\tau)$ for $t\in \mathbb{R}$.
For $\lambda_{[a,b]}\in \mathcal{MF}$ and $t\in \mathbb{R}$, $E_{t\lambda_{[a,b]}}$ defines a deformation of the marked torus $(M_\tau,f_\tau)\in \teich_1$ for any $\tau\in \mathbb{H}$. We call the deformation $E_{t\lambda_{[a,b]}}$ the \emph{(left) horocyclic deformation} at time $t$ along $\lambda_{[a,b]}$.
%\section{Wolpert's Duality}

The \emph{horcyclic vector field $\frac{\partial}{\partial t_{[a,b]}}$ determined by $\lambda_{[a,b]}\in \mathcal{MF}$} is a vector field on $\teich_1=\mathbb{H}$ 
 which is given by the infinitesimal horocyclic deformation along $\lambda_{[a,b]}$ at every point of $\hyperbolic$, that is, \begin{equation}
\left(\frac{\partial}{\partial t_{[a,b]}}\right)_\tau
=\left.\frac{dE_{t\lambda_{[a,b]}}(\tau)}{dt}\right|_{t=0}
\label{eq:infinitesimal_Earthquake_Def}
\end{equation}
for $\tau\in \mathbb{H}$.
We claim the following, which follows from the fact that every vector on $\hyperbolic$ is tangent to a horocycle, but we prove it more concretely by a calculation.
\begin{proposition}
\label{prop:earthquake_vector_homeo}
For any $\tau\in \mathbb{H}\cong\teich_1$,
the map
\begin{equation}
\label{eq:earthquake_vector_homeo}
\mathcal{MF}\ni \lambda_{[a,b]}\mapsto \left(\frac{\partial}{\partial t_{[a,b]}}\right)_\tau\in T_\tau\mathbb{H}=T_\tau\teich_1
\end{equation}
is a homeomorphism.
\end{proposition}

\begin{proof}
Since the map is evidently smooth with respect to the coordinates, we have only to show that it is bijective.
Under the period coordinate $\tau = u + iv \in \mathbb{H}$ of $\teich_1$, the earthquake vector field $\frac{\partial}{\partial t_{[a,b]}}$ is given by
\begin{align}
\left(\frac{\partial}{\partial t_{[a,b]}}\right)_\tau
&=
-\frac{(a + b\tau)^2}{\ell_{[a,b]}(\tau)}
= -\sqrt{{\rm Im}(\tau)} |a + b\tau| \left(\frac{a + b\tau}{|a + b\tau|}\right)^2.
\label{eq:infinitesimal_Earthquake_Def2}
\end{align}
Since the $\mathbb{R}$-linear map $(a,b) \mapsto a + b\tau$ is an isomorphism, as $(a,b)$ moves around any  circle centred at $(0,0)$ in $\mathbb{R}^2$, 
$$
\frac{a + b\tau}{|a + b\tau|}
$$
goes around the unit circle once. Hence, the map \eqref{eq:earthquake_vector_homeo} is surjective and proper.

We next show that the map \eqref{eq:earthquake_vector_homeo} is injective. Since ${\rm Im}(\tau) > 0$, the image of the map is $0$ if and only if $[a,b] = [0,0]$. 
Suppose that
$$
-\frac{(a + b\tau)^2}{\ell_{[a,b]}(\tau)} = -\frac{(c + d\tau)^2}{\ell_{[c,d]}(\tau)}
$$
for $[a,b], [c,d] \in \mathcal{MF} \setminus \{0\}$.  
Since $a + b\tau \ne 0$ and $c + d\tau \ne 0$, this equation is equivalent to
$$
\frac{(a + b\tau)^2}{(c + d\tau)^2} = \frac{|a + b\tau|}{|c + d\tau|}.
$$
This implies that $\frac{|a + b\tau|}{|c + d\tau|} = 1$, hence 
$(a + b\tau)^2 = (c + d\tau)^2$,  and $a + b\tau = \pm(c + d\tau)$. Since ${\rm Im}(\tau) > 0$, we have $(c,d) = \pm(a,b)$. Thus, we have shown that the map \eqref{eq:earthquake_vector_homeo} is injective. 

Therefore, the correspondence in \eqref{eq:earthquake_vector_homeo} is a homeomorphism.
\end{proof}

From \eqref{eq:length_function_ab}, the differential of the length function $\ell_{[a,b]}$ is equal to
$$
(d\ell_{[a,b]})_\tau=
\dfrac{b(a+bu)}{v\ell_{[a,b]}(\tau)}du
-
\frac{(a+bu)^2-b^2v^2}{2v^2\ell_{[a,b]}(\tau)}dv.
$$
Therefore, from \eqref{eq:WP_funcamental_form_uv} and \eqref{eq:infinitesimal_Earthquake_Def2}, we obtain the following \emph{Wolpert duality}
\begin{equation}
\label{eq:Wol_Dual}
\omega_{WP}\left(
\frac{\partial}{\partial t_{[a,b]}},
\cdot
\right)
=
d\ell_{[a,b]}
\end{equation}
for any $[a,b]\in \mathcal{MF}$.

We denote by $\nabla \ell_{[a,b]}$ the \emph{gradient vector field of $\ell_{[a,b]}$}
with respect to the Weil--Petersson metric $ds^2_{WP}$. By definition, the gradient vector field is given by
\begin{equation}
\label{eq:duality_gradient}
ds^2_{WP}(\nabla \ell_{[a,b]},\cdot)=d\ell_{[a,b]}.
\end{equation}
More explicitly, from \eqref{eq:WP_metric_uv}, the gradient vector field is expressed as
$$
(\nabla \ell_{[a,b]})_\tau
=
\dfrac{2vb(a+bu)}{\ell_{[a,b]}(\tau)}
\left(\frac{\partial}{\partial u}\right)_{\tau}
-
\frac{(a+bu)^2-b^2v^2}{\ell_{[a,b]}(\tau)}
\left(\frac{\partial}{\partial v}\right)_{\tau}
$$
under the period coordinate $\tau=u+iv\in \mathbb{H}\cong \teich_1$.
From \eqref{eq:infinitesimal_Earthquake_Def2}, we see that the gradient and the earthquake vector fields are related by
\begin{equation}
\label{eq:gradient_torus_dual}
(\nabla \ell_{[a,b]})_\tau
=J_0\left(
\left(\frac{\partial}{\partial t_{[a,b]}}\right)_\tau
\right),
\end{equation}
where $J_0$ is the complex structure
$$
J_0=-dv\otimes \frac{\partial}{\partial u}+du\otimes \frac{\partial}{\partial v}
$$
 on $\mathbb{H}\cong \teich_1$. Therefore, 
 \Cref{prop:earthquake_vector_homeo} implies that
\begin{equation}
\label{eq:gradient_vector_homeo}
\mathcal{MF}\ni \lambda_{[a,b]}\mapsto (\nabla \ell_{[a,b]})_\tau\in T_\tau\mathbb{H}=T_\tau\teich_1
\end{equation}
is also a homeomorphism.

\subsection{The horocyclic metric on $\mathbb{H}$}
By  \Cref{prop:earthquake_vector_homeo}, any tangent vector to $\teich_1$ is represented by a horocyclic vector.
We define the \emph{horocyclic norm} $\|\cdot\|_h$ on $T\teich_1$ by
$$
\left\|
\left(\frac{\partial}{\partial t_{[a,b]}}\right)_\tau
\right\|_h=\ell_{[a,b]}(\tau)
$$
for $\lambda_{[a,b]}\in \mathcal{MF}$ and $\tau\in \mathbb{H}\cong \teich_1$. The fact that $\Vert \cdot \Vert_h$ is a norm is a consequence of Theorem 3.1 below.

The following is our main result for the torus case.

\begin{theorem}
\label{thm:earthquake_norm}
We have the identity
$$
\|\cdot\|_h=2\,ds_\hyperbolic
$$
on $T\mathbb{H}=T\teich_1$,
where $ds_{\mathbb{H}}$ is the hyperbolic length element on $\mathbb{H}$ of curvature $-4$
(which is the Kobayashi metric on $\teich_1$).
\end{theorem}

\begin{proof}
From \eqref{eq:length_function_ab} and \eqref{eq:infinitesimal_Earthquake_Def2}, we have
\begin{align*}
ds_{\mathbb{H}}\left(\left(\frac{\partial}{\partial t_{[a,b]}}\right)_\tau\right)
&=\dfrac{1}{2v}\frac{|a+b\tau|^2}{\ell_{\lambda_{[a,b]}}(\tau)}=\dfrac{1}{2v}\dfrac{\sqrt{v}|a+b\tau|^2}{|a+b\tau|}=\frac{|a+b\tau|}{2\sqrt{v}}\\
&=\frac{1}{2}\ell_{[a,b]}(\tau)=\frac{1}{2}
\left\|
\left(\frac{\partial}{\partial t_{[a,b]}}\right)_\tau
\right\|_h.
\end{align*}
By \Cref{prop:earthquake_vector_homeo}, 
we have the desired coincidence.
\end{proof}

Analogously to the case of hyperbolic surfaces studied in \cite{HOPP}, there is  a duality between the tangent space with the horocycle norm and the cotangent space with  Thurston's metric conorm, which we denote by $\Vert . \Vert_{\rm Th}^*$, as follows.

\begin{theorem}
For each point $x$ in $\teich_1$, the linear map $\omega_{WP}$ defined in \cref{eq:Wol_Dual} is a linear  isometry from $(T_x \teich_1, \Vert . \Vert_h)$ to $(T^*_x \teich_1, \Vert . \Vert_{\rm Th}^*)$.
\end{theorem}
\begin{proof}
The map $\omega_{WP}$ takes the tangent vector $\left(\frac{\partial}{\partial t_{[a,b]}}\right)_\tau$, which has  horocyclic norm equal to $\ell_{[a,b]}(\tau)$, to the cotangent vector $d\ell_{[a,b]}$.
By exactly the same argument as Theorem 5.1 in \cite{ThM}, we see that the Thurston norm of the vector $d\ell_{[a,b]}$ is equal to $\ell_{[a,b]}(\tau)$.
This shows that the linear map $\omega_{WP}$ is an isometry.
\end{proof}

Since Thurston's metric on the Teichm\"{u}ller space of  area-$1$ tori is isometric to the hyperbolic metric on $\hyperbolic$ with curvature $-4$ (see the paper by Sa{\u g}lam, \cite{Sag}), this duality, unlike the case of the Teichm\"uller spaces of hyperbolic surfaces considered in \cite{HOPP}, corresponds to the canonical duality between the tangent and the cotangent spaces at each point of $\hyperbolic$.

\section{The general surface case}
In this section, we consider general surfaces $\Sigma_{g,m}$, and define the horocyclic deformations on their Teichm\"uller spaces using Teichm\"{u}ller discs.
We show that the horocyclic norm, defined in the same way as the torus case, induces  the same metric as the Teichm\"{u}ller metric on each Teichm\"{u}ller disc.
Since Teichm\"{u}ller discs are totally geodesic with respect to the Teichm\"{u}ller metric, this will imply that the horocyclic metric coincides with the Teichm\"{u}ller metric, which is the same situation as in the case where the surface is the torus.
Furthermore, we show that the Legendre transform gives a duality analogous to the duality via the Weil--Petersson form explained in \cite{HOPP}.
\subsection{Teichm\"uller discs}
Let $x=(M,f)$ be a point in $\teich_{g,m}$. 
For $q\in \mathcal{Q}_x\setminus\{0\}$ and $\zeta\in \mathbb{H}$, we define $(M_\zeta,f_\zeta)\in \teich_{g,m}$ as follows.
Consider the quasiconformal map $h_\zeta$ on $M$ whose Beltrami differential is $\frac{\zeta-i}{\zeta+i}  \frac{\overline{q}}{|q|}$. Set $M_\zeta=h_\zeta(M)$ and $f_\zeta=h_\zeta\circ f$.
Then, the map
$$
\Phi_q\colon \mathbb{H}\ni \zeta\mapsto \Phi_q(\zeta)=(M_\zeta,f_\zeta)\in \teich_{g,m}
$$
is holomorphic. This map and its image are called the \emph{Teichm\"uller disc} associated with the quadratic differential $q$ (see Marden--Masur \cite{MM}). Teichm\"uller discs are characterised as complex geodesics for the Teichm\"uller--Kobayashi Finsler structure in $(\teich_{g,m},d_T)$. 
Teichm\"uller discs vary continuously on $\mathcal{Q}_{x}$ in the sense that when $q_n\to q_0$ in $\mathcal{Q}_x$,
$\Phi_{q_n}$ converges to $\Phi_q$ uniformly on any compact set in $\mathbb{H}$.

\subsection{Extremal length}
Let $x=(M,f)$ be a point in $\teich_{g,m}$. With every $q\in \mathcal{Q}_x$, we associate a measured foliation $v(q)$, called the \emph{vertical foliation} of $q$ defined by
$$
i(\alpha,v(q))=\inf_{\alpha'\in \alpha}\int_{f(\alpha')}|{\rm Re}\sqrt{q}|
$$
for $\alpha\in \mathcal{S}$. Hubbard and Masur (\cite{HM}) showed that
the correspondence
$$
\mathcal{Q}_x\ni q\mapsto v(q)\in \mathcal{MF}
$$
is a homeomorphism. For $\lambda\in \mathcal{MF}$, we denote by $q_{\lambda,x}$ the holomorphic quadratic differential in $\mathcal{Q}_x$ with $v(q_{\lambda,x})=\lambda$. We call $q_{\lambda,x}$ the \emph{Hubbard--Masur differential} for $\lambda\in \mathcal{MF}$ on $x=(M,f)\in \teich_{g,m}$. 
Hubbard and Masur  also showed that $$
\mathcal{MF}\times \teich_{g,m}\ni (\lambda,x)\mapsto q_{\lambda,x}\in \mathcal{Q}_{g,m}
$$
is a homeomorphism.

We recall that the \emph{extremal length} of $\lambda\in \mathcal{MF}$ on $x=(M,f)\in \teich_{g,m}$ can be defined by
$$
{\rm Ext}_x(\lambda)=\|q_{\lambda,x}\|=\iint_{M}|q_{\lambda,x}(z)|dxdy\quad (z=x+iy).
$$
When we fix $\lambda \in\mathcal{MF}$ and regard the extremal length of $\lambda$ as a function on the Teichm\"uller space $\teich_{g,m}$, we write ${\rm Ext}_\lambda(x)$ to denote ${\rm Ext}_x(\lambda)$.

The extremal length function on $\teich_{g,m}\times \mathcal{MF}$ is continuous.
From the definition, we see that ${\rm Ext}_x(t\lambda)=t^2{\rm Ext}_x(\lambda)$ for $t\ge 0$. Gardiner showed in \cite{Ga} that the extremal length function is of class $C^1$ and that
\begin{equation}
\label{eq:Gardiner_formula}
(d\,{\rm Ext}_\lambda)_x[v]=-2{\rm Re}\langle v,q_{\lambda,x}\rangle
\end{equation}
for any $v\in T_x\teich_{g,m}$ and $x\in \teich_{g,m}$.

\subsection{Horocyclic deformations and Horocyclic norm}
Let $\lambda$ be a mesured foliation on $\Sigma_{g, m}$, and 
define $\ell_{\lambda}(x)={\rm Ext}_x(\lambda)^{1/2}$.

For $\zeta\in \mathbb{H}$, we define the \emph{horocyclic deformation} at time $t$ along $\lambda$ by
$$
\teich_{g,m}\ni x\mapsto E_{t\lambda}(x)=\Phi_{q_{\lambda,x}}(-t\ell_\lambda(x)+i).
$$
As is shown in  \cite{MM}, if  $\lambda$ lies in $\mathcal{S}$ and $t=\ell_\lambda(x)$, then $E_{\ell_{\lambda}(x)\lambda}(x)$ is the image of $x$ under the action of the (right-handed) Dehn twist around $\lambda$.
The \emph{horocyclic vector field} determined by $\lambda$ is a vector field on $\teich_{g,m}$ defined by
$$
\left(\frac{\partial}{\partial t_\lambda}\right)_x
=\left(
\frac{d  E_{t\lambda}(x)}{dt}
\right)_{t=0}
$$
for every point $x \in \teich_{g,m}$.
\begin{proposition}
\label{prop:horocyclic-vector-field}
For $x\in \teich_{g,m}$, we have
$$
\left(\frac{\partial}{\partial t_\lambda}\right)_x=\left[
\frac{i\ell_\lambda(x)}{2}\frac{\overline{q_{\lambda,x}}}{|q_{\lambda,x}|}
\right]\in T_x\teich_{g,m}.
$$
In particular, the horocyclic vector field is a continuous vector field on $\teich_{g,m}$.
\end{proposition}

\begin{proof}
Indeed, since
$$
\frac{(-t\ell_\lambda(x)+i))-i}{(-t\ell_\lambda(x)+i)+i}=\frac{i\ell_\lambda(x)}{2}t+o(t)
$$
as $t\to 0$, the equality follows from the definition of a Teichm\"uller disc.
\end{proof}

From Teichm\"uller's theorem and \Cref{prop:horocyclic-vector-field}, each non-zero tangent vector to $\teich_{g,m}$ at a point $x$ is represented by a horocyclic vector up to a scalar.
We define the \emph{horocyclic norm} on $T_x\teich_{g,m}$ by
$$
\left\|\left(\frac{\partial}{\partial t_\lambda}\right)_x\right\|_h=\ell_\lambda (x).
$$
From \eqref{eq:teichmuller-beltrami}, we have the following theorem, which is a complete analogue of \Cref{thm:earthquake_norm}.
\begin{theorem}
\label{thm:Horocyclic_norm}
We have the following identity
$$
\left\|\,\cdot\, \right\|_h
%=\frac{1}{2}
=2\,\kappa\left(\,\cdot\,
%\left.\frac{\partial}{\partial t_\lambda}\right|_x
\right)
$$
on $T\teich_{g,m}$.
\end{theorem}

\subsection{An analogue of Wolpert's duality for general surfaces}
In this subsection, we prove a duality result which is analogous to the duality between the infinitesimal earthquakes and the differentials of hyperbolic length functions  studied in \cite{HOPP}.
The Legendre transform in  Finsler geometry was introduced and developed by Ohta--Sturm in \cite{Ohta}.
Here, we consider a complex version of this Lengendre transform.
Let $M$ be a complex manifold and $F$ a complex Finsler metric. The latter is an assignment of a complex norm $F_x$ on the holomorphic tangent space $T_xM$ for each $x\in M$.
The associated \emph{conorm} $F^*_x$ on the holomorphic cotangent space $T^*_x\!M$ is defined by
\begin{equation}
\label{eq:conorm}
F^*_x\!(\alpha)=\sup\{{\rm Re}(\alpha(\xi))\mid F_x(\xi)\le 1\}.
\end{equation}
We assume that $F$ is strictly convex in the sense that for $\xi_1,\xi_2\in T_xM$ with $F_x(\xi_1)=F_x(\xi_2)=1$, $F_x((\xi_1+\xi_2)/2)<1$.
Then, for each $x\in M$, the (complex) \emph{Legendre transform} $J_x\colon T^*_x\!M\to T_xM$ with respect to $F$ is defined to be a map assigning to $\alpha$  the (unique) maximiser of the functional
\begin{equation}
\label{eq:LT}
T_xM\ni \xi\mapsto {\rm Re}(\alpha(\xi))-\frac{1}{2}F_x(\xi)^2-\frac{1}{2}F^*_x\!(\alpha)^2.
\end{equation}
The strictly convexity guarantees the uniqueness of the maximiser. %Ohta and Sturm discuss the Legendre transforms for real Finsler manifolds (cf. \cite[\S1.2]{Ohta}). The Legendre transform \eqref{eq:LT} discussed here is naturallly thought of the complexification of that given in \cite{Ohta}.

We now return to the Teichm\"uller space. 
For $x\in \teich_{g,m}$, the \emph{infinitesimal Teichm\"uller homeomorphism} 
is defined by
\begin{equation}
\label{eq:TB}
\mathcal{Q}_x\ni q\mapsto \left[\|q\|\dfrac{\overline{q}}{|q|}\right]\in T_x\teich_{g,m}
\end{equation}
when $q\ne 0$, and $J_x(0)=0$ otherwise. From the viewpoint of  Finsler geometry, we see that the infinitesimal Teichm\"uller homeomorphism is nothing but the Legendre transform as follows.

\begin{proposition}[Legendre transform with respect to the Teichm\"uller metric]
\label{prop:Legendre_transform}
For $x\in \teich_{g,m}$,  the infinitesimal Teichm\"{u}ller homeomorphism  coincides with the Legendre transform with respect to the Teichm\"uller metric.
\end{proposition}

\begin{proof}
Royden noted in \cite{Roy} that the Teichm\"uller norm is strictly convex.
By \eqref{eq:teichmuller-beltrami}, the conorm of the Teichm\"uller norm is the $L^1$-norm on $\mathcal{Q}_x$.
Therefore, for $q\in \mathcal{Q}_x$, its Legendre transform is by definition the maximiser of the functional
\begin{equation}
\label{eq:functional}
T_x\teich_{g,m}\ni v\mapsto {\rm Re}\langle v,q\rangle-\frac{1}{2}\kappa(v)^2-\frac{1}{2}\|q\|^2.
\end{equation}
Since
$$
{\rm Re}\langle v,q\rangle-\frac{1}{2}\kappa(v)^2-\frac{1}{2}\|q\|^2\le {\rm Re}\langle v,q\rangle-\kappa(v)\|q\|\le 0,
$$
the functional in \eqref{eq:functional}  attains the maximum $0$ when $v=\|q\|\overline{q}/|q|$.
By Teichm\"uller's uniqueness theorem, we see that ${\rm Re}\langle v,q\rangle-\kappa(v)\|q\|=0$ if and only if $v=\|q\|\overline{q}/|q|$.
Thus, we have shown that the  infinitesimal Teichm\"{u}ller homeomorphism \eqref{eq:TB} coincides with the Legendre transform.
\end{proof}

From \eqref{eq:Gardiner_formula},
$$
(\partial \ell_{\lambda})_x=\frac{1}{2\ell_\lambda}(\partial{\rm Ext}_\lambda)_x=-\frac{1}{2\ell_\lambda}q_{\lambda,x}
\in \mathcal{Q}_x\cong T^*_x\!\teich_{g,m}
$$
for $x\in \teich_{g,m}$ and $\lambda\in \mathcal{MF}$.
Hence, we have the following.
\begin{theorem}[Duality]
\label{thm:duality}
Let $J_x\colon T_x^*\!\teich_{g,m}\to T_x\teich_{g,m}$ be the Legendre transform with respect to the Teichm\"uller metric.
For $\lambda \in\mathcal{MF}$ and $x\in \teich_{g,m}$,
$$
J_x\left(
(\partial \ell_\lambda)_x
\right)=i\left(\frac{\partial}{\partial t_\lambda}\right)_x.$$
In particular, we have $\kappa^*(\alpha)=\Vert J_x(\alpha)\Vert_h$ for every $\alpha \in T^*_{g,m}$, where $\kappa^*$ is the dual norm of $\kappa$.
\end{theorem}

\begin{proof}
Since $\|q_{\lambda,x}\|={\rm Ext}_x(\lambda)=\ell_\lambda(x)^2$,
from \Cref{prop:horocyclic-vector-field}, we get
\begin{align*}
J_x\left(
(\partial \ell_\lambda)_x
\right)
&=\left[\frac{1}{2\ell_\lambda(x)}\|q_{\lambda,x}\|\dfrac{\overline{(-1/2\ell_\lambda(x))q_{\lambda,x}}}{|(-1/2\ell_\lambda(x))q_{\lambda,x}|}\right] \\
&=-\left[\frac{\ell_\lambda(x)}{2}\frac{\overline{q_{\lambda,x}}}{|q_{\lambda,x}|}\right]
=i\left(\frac{\partial}{\partial t_\lambda}\right)_x,
\end{align*}
and we are done.
\end{proof}

\begin{remark}
As is shown in  \cite[\S1.4]{Ohta}, in the case when the manifold is Riemannian, the Legendre transform takes the differential of a smooth function to its gradient. From this point of view, our duality formula in \cref{thm:duality} is a complete analogue of the Wolpert duality \eqref{eq:gradient_torus_dual} in the torus case.

%In the case of the Hermitian metric, 
%the Legendre transform assigns the gradient vectors to smooth (real-valued) functions.
%Indeed, let $M$ be a complex manifold, $h=\{h_x\}_{x\in M}$ an Hermitian metric on $M$, and $g_x=2{\rm Re}(h_x)$ the associated Riemannian metric. Let $F_x(\xi)=\sqrt{g_x(\xi,\xi)}$.
%
%Let $H$ be a smooth function defined around $x\in M$. The gradient vector $(\nabla H)_x\in T_xM$ of $H$ at $x\in M$ satisfies
%$$
%g_x(Y,(\nabla H)_x)=\partial H(Y)
%$$
%for $Y\in T_xM$.
%
%of The conorm \eqref{eq:conorm} of $(\partial H)_x$ for $F_x$ is attained at the gradient vector $(\nabla H)_x$ of $H$
%and satisfies $F_x((\nabla H)_x)=F^*((\partial H)_x)$.
%Hence, the functional \eqref{eq:LT} for $F_x$ is also attained at the $(\nabla H)_x\in T_xM$ for the covector $(\partial H)_x$. Thus, the Legendre transform for $F_x=(g_x)^{1/2}$ is given by
%$$
%T^*_x\! M\ni (\partial H)_x\mapsto (\nabla H)_x\in T_xM.
%$$
%
%In general, the fundamental form $\omega$ of the Hermite metric $h_x$ defined by
%$\omega(X,Y)=-2{\rm Im}(h_x(X,Y))=i(h_x(X,Y)-\overline{h_x(X,Y)})$ for $X$, $Y\in T_xM$.
%

\end{remark}

\bibliography{horocyclic}
\bibliographystyle{acm}

%\begin{thebibliography}{99}
%
%\bibitem{Gardiner}
%F.~P.~Gardiner,
%\newblock Measured foliations and the minimal norm property for quadratic differentials,
%\newblock Acta Math. {\bf 152}, 57--76 (1984).
%
%\bibitem{IT}
%Y.~Imayoshi  and M.~Taniguchi,
%\newblock An introduction to {T}eichm{\"u}ller spaces
%\newblock Springer-Verlag, Tokyo (1992).
%
%\bibitem{MM}
%A.~Marden and H.~Masur,
%\newblock A foliation of {T}eichm{\"u}ller space by twist invariant disks,
%\newblock Math. Scand. {\bf 2}, 211--228 (1975).
%
%\bibitem{Royden}
%H.~L.~Royden,
%\newblock
%Automorphisms and isometries of {T}eichm{\"u}ller space,
%\newblock
%Advances in the {T}heory of {R}iemann {S}urfaces ({P}roc. {C}onf., {S}tony {B}rook, {N}.{Y}., 1969),
%\newblock
%Ann. of Math. Studies, No. {\bf 66}. Princeton Univ. Press, Princeton, N.J. 369--383 (1971).
%\end{thebibliography}

\end{document}